\documentclass[10pt, leqno]{amsart}
\usepackage{amscd}
\usepackage{amssymb}
\usepackage{amsmath}
\usepackage{amsfonts}
\usepackage{stmaryrd}
\usepackage{amsthm}
\usepackage[cmtip,all]{xypic}
\usepackage{pifont}

\usepackage[usenames, dvipsnames]{color}

\theoremstyle{plain}
\newtheorem{Lem}{Lemma}[section]

\newtheorem{Thm}[Lem]{Theorem}

{\theoremstyle{definition} 

\newtheorem{Rk}[Lem]{Remark}
\newtheorem{Def}[Lem]{Definition}}

\newcommand{\zig}{\addtocounter{Lem}{1}\tag{\theLem}}

\def\:{\colon}
\DeclareMathOperator*{\colim}{colim}

\DeclareMathOperator*{\lims}{lim^\mathit{s}}

\begin{document}
\title{Several homotopy fixed point spectral sequences in 
telescopically localized algebraic $K$-theory}
\author{Daniel G. Davis}
\address{University of Louisiana at Lafayette, Department of Mathematics, Lafayette, Louisiana, USA}
\email{dgdavis@louisiana.edu}
\begin{abstract}
Let $n \geq 1$, $p$ a prime, and $T(n)$ any representative of the 
Bousfield class of the telescope $v_n^{-1}F(n)$ of a finite type $n$ complex. 
Also, let 
$E_n$ be the Lubin-Tate spectrum, $K(E_n)$ its algebraic $K$-theory spectrum, and $G_n$ the extended Morava stabilizer group, a profinite group. Motivated by an Ausoni-Rognes conjecture, 
we show that there are two spectral sequences
\[{^{I}}\mspace{-3mu}E_2^{s,t} \Longrightarrow \pi_{t-s}((L_{T(n+1)}K(E_n))^{hG_n})
\Longleftarrow {^{II}}\mspace{-2mu}E_2^{s,t}\] with common abutment $\pi_\ast(-)$ of the continuous homotopy fixed points of $L_{T(n+1)}K(E_n)$, where 
${^{I}}\mspace{-3mu}E_2^{s,t}$ is continuous cohomology with coefficients in a certain 
tower of discrete $G_n$-modules. If the tower satisfies 
the Mittag-Leffler condition, then there are 
continuous cochain cohomology groups \[{^{I}}\mspace{-3mu}E_2^{\ast,\ast} \cong H^\ast_\mathrm{cts}(G_n, \pi_\ast(L_{T(n+1)}K(E_n))) \cong 
{^{II}}\mspace{-2mu}E_2^{\ast,\ast}.\] We isolate 
two hypotheses, the first of which is true when $(n,p) = (1,2)$, that imply 
$(L_{T(n+1)}K(E_n))^{hG_n} \simeq L_{T(n+1)}K(L_{K(n)}S^0)$. Also, we show 
that there is a spectral sequence
\[H^s_\mathrm{cts}(G_n, \pi_t(K(E_n) \otimes T(n+1))) \Longrightarrow 
\pi_{t-s}((K(E_n) \otimes T(n+1))^{hG_n}).\]
\end{abstract}

\maketitle
\section{Introduction}
\subsection{The basic characters in this work and a presentation of $L_{T(n+1)}K(E_n)$ as a $G_n$-spectrum}\label{oneone}
Let $n \geq 1$, let $p$ be any prime, and let $\mathrm{Sp}$ be the symmetric monoidal $\infty$-category of spectra. 
Let $F(n)$ denote a finite type $n$ complex and let $v_{n}^{-1}F(n)$ be the telescope of a $v_{n}$-self-map $v$ on $F(n)$. By \cite[Lemma 4]{global} and \cite[page 103]{Mahowald/Sadofsky}, $v_n^{-1}F(n)$ is independent of the choice of $v$ and the Bousfield class of $v_{n}^{-1}F(n)$ is independent of the choices for $F(n)$ and $v$. As is common, we let $T(n)$ denote a representative of this Bousfield class. 
Also, given a ring spectrum $A$ (that is, an algebra in $\mathrm{Sp}$; also called an $\mathbb{E}_1$-ring spectrum), $K(A)$ denotes 
the algebraic $K$-theory spectrum of $A$. 

Let $K(n)$ be the $n$th Morava $K$-theory spectrum and 
let $E_n$ be the Lubin-Tate spectrum with 
\[\pi_\ast(E_n) = W(\mathbb{F}_{p^n})[[u_1, ..., u_{n-1}]][u^{\pm 1}],\] where 
the complete power series ring over the Witt vectors for the field with $p^n$ elements is in 
degree $0$ and $|u| = -2$. Also, $G_n$ denotes the $n$th extended Morava stabilizer group, a profinite group which acts on $E_n$ by maps of commutative algebras (for 
more background on this, see \cite{DH, 
Pgg/Hop0, AndreQuillen}). As in \cite[(1.4)]{DH}, let 
\[G_n = U_0 \supsetneq U_1 \supsetneq \cdots \supsetneq U_i \supsetneq \cdots\] be a 
descending chain of open normal subgroups of $G_n$, with $\bigcap_{i \geq 0} U_i = \{e\}$. 
Then by \cite[Theorem 1]{DH}, \cite[Section 8]{joint}, \cite{Fausk}, and \cite[Theorem 1.2]{QuickLT}, 
there is a diagram 
\begin{equation}\label{diagram}\zig
E_n^{hG_n} = E_n^{hU_0} \to E_n^{hU_1} \to \cdots \to E_n^{hU_i} \to \cdots\end{equation} in the subcategory $\mathrm{CAlg}(\mathrm{Sp})$ of 
commutative algebras in $\mathrm{Sp}$ consisting of continuous homotopy fixed point 
spectra.

As in \cite[Proposition 4.22; Proposition 7.10, (d)]{HS}, we let 
\[M_0 \leftarrow M_1 \leftarrow \cdots \leftarrow M_j \leftarrow \cdots \] 
be a tower of generalized 
Moore spectra: each $M_j$ is finite of type $n+1$ and an atomic $\mu$-spectrum and the tower has 
the property that there is an equivalence
\[L_{F(n+1)} Z \simeq \lim_{j \geq 0} (Z \otimes M_j),\] for any spectrum $Z$. Let $MU_{(p)}$ denote the $p$-localization of $MU$. For each 
$j$, there is a sequence $a(j)_0, a(j)_1, ..., a(j)_{n}$ of powers of $p$ such that 
\[(MU_{(p)})_\ast(M_j) \cong (MU_{(p)})_\ast/(v_0^{a(j)_0}, v_1^{a(j)_1}, ..., v_{n}^{a(j)_{n}})\] and it is common 
to write 
\[M_j =M(v_0^{a(j)_0}, v_1^{a(j)_1},..., v_{n}^{a(j)_{n}}).\] 

As recalled in Definition \ref{finitelocalize} (see \cite[Definition 2.6]{lifeafter}, \cite{Mahowald/Sadofsky}), we use $L_{n+1}^f$ to denote the Bousfield localization functor that is often referred to as ``finite $E(n+1)$-localization" 
\cite{finitebousfield}. We can now state a result that we use in Section \ref{1.2} to define the continuous $G_n$-homotopy fixed points of 
$L_{T(n+1)}K(E_n)$. 

\begin{Thm}\label{gettingtheballrolling}
For every $n \geq 1$ and all primes $p$, there is an equivalence
\[L_{T(n+1)}K(E_n) \simeq 
\lim_{j \geq 0} (\colim_{i \geq 0} (K(E_n^{hU_i}) \otimes L_{n+1}^f M_j)).\]
\end{Thm}

On the right-hand side of the above equivalence -- and elsewhere in this paper -- we use ``$\colim$" to denote the colimit in $\mathrm{Sp}$; any colimit in a subcategory of $\mathrm{Sp}$ with additional structure is marked as such. The proof of Theorem \ref{gettingtheballrolling}, which is in Section \ref{proofofrollingball}, uses 
a deep result from \cite{LMMT} in a key way. 

The $G_n$-action on $E_n$ induces a natural $G_n/U_i$-action on each 
$E_n^{hU_i}$, and hence, for each $i$ and $j$, $K(E_n^{hU_i}) \otimes L_{n+1}^f M_j$ 
has a natural $G_n/U_i$-action, with $G_n/U_i$ acting trivially on $L_{n+1}^f M_j$ (and 
diagonally on the smash product). It follows that in Theorem \ref{gettingtheballrolling}, the equivalence is $G_n$-equivariant. 

\begin{Def}
Let $G$ be any profinite group. As in \cite{ClausenMathewInventiones}, let 
$\mathrm{Sh}(\mathcal{T}_G, \mathrm{Sp})$ denote the $\infty$-category of $\mathrm{Sp}$-valued sheaves 
on the Grothendieck site $\mathcal{T}_G$ of finite continuous $G$-sets. Also, as in [op. cit.], 
$\mathrm{PSh}_{\Pi}(\mathcal{T}_G,\mathrm{Sp})$ denotes the $\infty$-category of presheaves of spectra on $\mathcal{T}_G$ that send finite coproducts in $\mathcal{T}_G$ to finite products (see also \cite[Section 2]{tmfbook}), 
and $\mathrm{PSh}(\mathcal{O}_G, \mathrm{Sp})$ is the $\infty$-category of presheaves 
on the orbit category of $G$ consisting of quotients of $G$ by open subgroups.
\end{Def}

By \cite[Theorem 1]{DH} and \cite[Section 8]{joint}, there is a presheaf $F$ in 
$\mathrm{PSh}(\mathcal{O}_{G_n}, \mathrm{Sp})$ defined by 
\[F \: G_n/U \mapsto E_n^{hU}, \ \ \ \text{$U$ an open subgroup of $G_n$,}\] 
where, as before, $E_n^{hU}$ is the continuous homotopy fixed point spectrum and $F$ 
is actually a presheaf of commutative algebras. Thus, there 
is a diagram $\{\mathcal{P}_j\}_{j \geq 0}$ in 
$\mathrm{PSh}(\mathcal{O}_{G_n}, \mathrm{Sp})$, with each presheaf $\mathcal{P}_j$ 
defined by
\[\mathcal{P}_j \: G_n/U \mapsto K(E_n^{hU}) \otimes L_{n+1}^fM_j, \ \ \ U \ \text{an open subgroup of} \ G_n.\] 

\subsection{Continuous homotopy fixed points for $L_{T(n+1)}K(E_n)$ and two associated homotopy fixed point spectral sequences}\label{1.2}
We say that a profinite group $G$ has ``finite cohomological dimension" if there is some integer $s_0$ such that the continuous cohomology $H^s_c(G,M) = 0$, whenever $s > s_0$ and $M$ is any discrete $G$-module. Recall that $G_n$ has finite virtual cohomological dimension: that is, $G_n$ contains an open subgroup $U$ of 
finite cohomological dimension. Also, 
as recalled (with more generality) in Section \ref{ctshfps}, if 
\[X_0 \to X_1 \to \cdots \to X_i \to \cdots\] is a diagram of $G_n$-spectra such that for 
each $i$, the $G_n$-action on $X_i$ factors through $G_n/U_i$ (that is, $U_i$ acts trivially on $X_i$), then there is the continuous homotopy fixed point spectrum $(\colim_{i \geq 0} X_i)^{hG_n}$, given by the totalization of an $\{i \geq 0\}$-indexed colimit of certain familiar 
cosimplicial spectra.

It is now natural to make the following definition, which follows a familiar template in algebraic 
$K$-theory (for example, see \cite[Proposition 3.1.2, last paragraph of 3.1, proof of Theorem 4.2.6]{Geisser}).

\begin{Def}\label{fun}
Let $n \geq 1$ and let $p$ be any prime. 
There is the continuous homotopy fixed point spectrum
\[
(L_{T(n+1)}K(E_n))^{hG_n} := \Bigl(\,\lim_{j \geq 0} \bigl(\colim_{i \geq 0} (K(E_n^{hU_i}) \otimes L_{n+1}^f M_j)\bigr)\mspace{-0mu}\Bigr)^{\negthinspace hG_n},
\] where the right-hand side is given by
\[\Bigl(\,\lim_{j \geq 0} \bigl(\colim_{i \geq 0} (K(E_n^{hU_i}) \otimes L_{n+1}^f M_j)\bigr)\mspace{-0mu}\Bigr)^{\negthinspace hG_n} \mspace{-6mu}:= 
\lim_{j \geq 0} \bigl(\colim_{i \geq 0} (K(E_n^{hU_i}) \otimes L_{n+1}^f M_j)\bigr)^{hG_n}.\] 
For each $j \geq 0$,  
by Theorem \ref{sheafone}, there is an equivalence
\[\bigl(\colim_{i \geq 0} (K(E_n^{hU_i}) \otimes L_{n+1}^f M_j)\bigr)^{hG_n} 
= (\colim_{i \geq 0} \mathcal{P}_j(G_n/U_i))^{hG_n} \simeq \widetilde{\mathcal{P}_j}_{\scriptscriptstyle{\prod}}(\ast)\] with the global sections of the 
Postnikov sheafification $\widetilde{\mathcal{P}_j}_{\scriptscriptstyle{\prod}}$ of the presheaf 
${\mathcal{P}_j}_{\scriptscriptstyle{\prod}}$ in $\mathrm{PSh}_{\scriptscriptstyle{\prod}}(\mathcal{T}_{G_n}, \mathrm{Sp})$ 
that is induced by $\mathcal{P}_j$. 
\end{Def}

To state the next result, whose proof is in Section \ref{sectiondiscrete}, we need the 
following notation. 
Given a profinite group $G$ and a tower $\{A_j\}_{j \geq 0}$ of discrete $G$-modules, we let 
$H^\ast_\mathrm{cont}(G; \{A_j\}_{j \geq 0})$ denote continuous cohomology in the sense of Jannsen \cite{Jannsen} and $H^\ast_\mathrm{cts}(G, \lim_{j \geq 0} A_j)$ is continuous cochain 
cohomology with coefficients in the stated topological $G$-module. Also, if $X^\bullet$ is a cosimplicial spectrum, then we let ${^{II}}\mspace{-2mu}E_2^{s,t}$ be the $E_2$-term of the associated homotopy spectral sequence
\[{^{II}}\mspace{-2mu}E_2^{s,t} := \lims_\Delta \pi_t(X^\bullet) \Longrightarrow \pi_{t-s}(\mathrm{Tot}(X^\bullet)).\]
 
\begin{Thm}\label{spectralsequences}
Let $n \geq 1$ and let $p$ be any prime. There are conditionally 
convergent homotopy fixed point spectral sequences 
\[{^{I}}\mspace{-3mu}E_2^{s,t} = H^s_\mathrm{cont}(G_n; \{\pi_t(K(E_n)
\otimes v_{n+1}^{-1}M_j)\}_{j \geq 0}) \Longrightarrow \pi_{t-s}((L_{T(n+1)}K(E_n))^{hG_n})\] 
and  
\[{^{II}}\mspace{-2mu}E_2^{s,t} \Longrightarrow \pi_{t-s}((L_{T(n+1)}K(E_n))^{hG_n}).\] If 
the tower $\{\pi_t(K(E_n) \otimes v_{n+1}^{-1}M_j)\}_{j \geq 0}$ satisfies 
the Mittag-Leffler condition for every $t \in \mathbb{Z}$, then for all $s \geq 0$, there are isomorphisms
\[{^{I}}\mspace{-3mu}E_2^{s,\ast} \cong H^s_\mathrm{cts}(G_n, \pi_\ast(L_{T(n+1)}K(E_n))) \cong 
{^{II}}\mspace{-2mu}E_2^{s,\ast},\] where for each $t$,  $\pi_t(L_{T(n+1)}K(E_n)) \cong 
\lim_{j \geq 0} \pi_t(K(E_n) \otimes v_{n+1}^{-1}M_j)$.   
\end{Thm}

For any spectrum $X$ with trivial $G_n$-action, there are two homotopy fixed point spectral sequences for $\pi_\ast((L_{K(n)}(E_n \wedge X))^{hG_n})$ that correspond to the two 
in Theorem \ref{spectralsequences}, but by \cite[Theorem 1.2]{jointwithTakeshi}, it is the second 
one, with its own particular ${^{II}}\mspace{-2mu}E_2^{\ast,\ast}$, that is isomorphic to the strongly convergent $K(n)$-local $E_n$-Adams 
spectral sequence for $\pi_\ast(L_{K(n)}X)$. One ingredient in the construction of this 
Adams-type spectral sequence is that $E_n$ is a commutative algebra. Similarly, 
$K(E_n)$ is a commutative algebra and it seems plausible that, in general, the second spectral sequence in Theorem \ref{spectralsequences} has better properties. Also, by \cite[Section 4.6]{joint}, it could happen that there are cases where 
${^{II}}\mspace{-2mu}E_2^{s,t}$ in Theorem \ref{spectralsequences} is equal 
to the continuous cochain cohomology group stated in the theorem, without the underlying tower of 
discrete $G_n$-modules satisfying the Mittag-Leffler condition. 

Since $F$ can be regarded as a presheaf of commutative algebras, there is 
the presheaf $L_{n+1}^f \circ K \circ F$ in $\mathrm{PSh}(\mathcal{O}_{G_n}, \mathrm{Sp})$ 
given by
\[G_n/U \mapsto L_{n+1}^fK(F(G_n/U)) = L_{n+1}^fK(E_n^{hU}), 
\ \ \ U \ \text{an open subgroup of} \ G_n,\]
which induces the diagram 
\[L_{n+1}^f K(E_n^{hU_0}) \to L_{n+1}^f K(E_n^{hU_1}) \to \cdots 
\to L_{n+1}^f K(E_n^{hU_i}) \to \cdots,\] with $G_n/U_i$ acting on $L_{n+1}^f K(E_n^{hU_i})$ for each $i$, and hence, there is the continuous homotopy fixed point 
spectrum \[(\colim_{i \geq 0} L_{n+1}^f K(E_n^{hU_i}))^{hG_n}.\] The next result, whose proof is in Section \ref{sectiondiscrete}, uses this last 
spectrum to show that $(L_{T(n+1)}K(E_n))^{hG_n}$ is $T(n+1)$-local. 

\begin{Thm}\label{discrete}
For each $n \geq 1$ and all primes $p$, there is an equivalence 
\[(L_{T(n+1)}K(E_n))^{hG_n} \simeq L_{T(n+1)}((\colim_{i \geq 0} L_{n+1}^f K(E_n^{hU_i}))^{hG_n}).\]
\end{Thm} 

\subsection{Potential relationships of $(L_{T(n+1)}K(E_n))^{hG_n}$ 
with an Ausoni-Rognes conjecture involving $K(E_n)^{hG_n}$}
The spectral sequences in Theorem \ref{spectralsequences} remind 
one of an Ausoni-Rognes conjecture (\cite[(0.1)]{acta}, \cite[page 46; Remark 10.8]{jems}; also, see the closely related \cite[Conjecture 4.2]{AusoniRognesGuido}), which states that (a) the $K(n)$-local unit map $\eta \: L_{K(n)}S^0 \to E_n$ induces a map
\[K(L_{K(n)}S^0) \to K(E_n)^{hG_n},\] which -- in this paper -- we refer to as $K_\tau$ and 
whose target 
$K(E_n)^{hG_n}$ is a homotopy fixed point spectrum whose construction is compatible with the profinite topology on $G_n$; 
and (b) the map
\[K_\tau \otimes T(n+1) 
\: K(L_{K(n)}S^0) \otimes T(n+1) \to K(E_n)^{hG_n} 
\otimes T(n+1)\] is an equivalence, so that $K_\tau$ is a $T(n+1)$-equivalence. 

\begin{Rk} 
The $\tau$ in $K_\tau$ is for ``transposing," since the $K(n)$-local unit 
\[E_n^{hG_n} \xleftarrow{\,\simeq\,} L_{K(n)}S^0\] of commutative algebras is an 
equivalence, by \cite[Theorem 1]{DH}, giving 
\[K(E_n^{hG_n}) \xleftarrow{\,\simeq\,} K(L_{K(n)}S^0),\] which implies that $K_\tau$ switches 
$K(-)$ and $(-)^{hG_n}$ in the case of $E_n$. 
\end{Rk}

\begin{Rk}
Currently, for every $n$ and $p$, there is not a published construction of $K(E_n)^{hG_n}$ or 
the map $K_\tau$. In \cite[Remark 1.5]{padicspectra2}, it is noted that according to Jacob Lurie, 
the condensed mathematics 
of Dustin Clausen and Peter Scholze can be used to define 
$K(E_n)$ as a condensed spectrum, and then building on this, there is a candidate 
definition of $K(E_n)^{hG_n}$ in the condensed setting. Similarly (see [op. cit.]),  by viewing $K(E_n)$ as a 
pyknotic spectrum \cite[Section 3.1]{pyknotic}, 
there should be a pyknotic version of the ``condensed candidate" for $K(E_n)^{hG_n}$. 
\end{Rk}

\begin{Rk}
Strictly speaking, the Ausoni-Rognes conjecture referred to above predicts that $K_\tau \otimes v_{n+1}^{-1}F(n+1)$ is an equivalence, but this is logically the same as $K_\tau \otimes T(n+1)$ being an 
equivalence.
\end{Rk}

From the commutative square 
\[\xymatrix@=.5in{
K(L_{K(n)}S^0) \ar[d] \ar[r]^-{\scriptscriptstyle{K_\tau}} & K(E_n)^{hG_n} \ar[d]\\ 
L_{T(n+1)}K(L_{K(n)}S^0) \ar[r]^-{\scriptscriptstyle{L_{T(n+1)}K_\tau}} & 
L_{T(n+1)}(K(E_n)^{hG_n}),}\] 
in which the vertical maps, as the usual localizations, are 
$T(n+1)$-equivalences, we see that if part (a) of the above Ausoni-Rognes conjecture holds, 
then part (b) holds 
if and only if $L_{T(n+1)}(K_\tau)$ is an equivalence. This leads one to 
wonder about the relationship between \[L_{T(n+1)}(K(E_n)^{hG_n}) \ \text{and} \ 
(L_{T(n+1)}K(E_n))^{hG_n}\] and if there is an equivalence between 
\[(L_{T(n+1)}K(E_n))^{hG_n} \ \text{and} \ L_{T(n+1)}K(L_{K(n)}S^0).\] 

The following result, whose proof is in Section \ref{strategy}, gives two 
hypotheses that when jointly satisfied imply that 
the last two spectra above are equivalent. If a finite group 
$H$ acts on a spectrum $Z$, we let $\prod_{H^\bullet} Z$ denote the induced cosimplicial spectrum; see Section \ref{ctshfps} for more detail. Also, for each $i \geq 0$, we let 
\[K(\eta \scriptstyle{\nearrow} \mspace{1.5mu}\displaystyle{}i) \: L_{T(n+1)}K(E_n^{hG_n}) \to (L_{T(n+1)}K(E_n^{hU_i}))^{hG_n/U_i}\] be the natural map.

\begin{Thm}\label{wouldgiveprogress}
Let $n \geq 1$ and let $p$ be a prime. If
\begin{itemize}
\item[{(H1)}]
the map
\[\colim_{i \geq 0} K(\eta \scriptstyle{\nearrow} \mspace{1.5mu}\displaystyle{}i) \: 
L_{T(n+1)}K(E_n^{hG_n}) \to \colim_{i \geq 0} (L_{T(n+1)}K(E_n^{hU_i}))^{hG_n/U_i}\] 
is a $T(n+1)$-equivalence, and 
\item[{(H2)}]
for each $j \geq 0$, the canonical map
\[\colim_{i \geq 0} \displaystyle{} \mathrm{Tot}(C(n,j, i)^\bullet)
\to \mathrm{Tot}(\colim_{i \geq 0} C(n,j,i)^\bullet)\] is an equivalence, 
where for each $i$, 
\[C(n,j,i)^\bullet := \textstyle{\prod}_{(G_n/U_i)^\bullet} (K(E_n^{hU_i}) \otimes L_{n+1}^fM_j),\] 
\end{itemize} 
then there is an equivalence 
\[(L_{T(n+1)}K(E_n))^{hG_n} \simeq L_{T(n+1)}K(L_{K(n)}S^0).\] 
\end{Thm} 

When $n=1$ and $p=2$, Remark \ref{commentone} shows that (H1) is true, but the validity of 
(H2) is still open. For pairs $(n,p) \neq (1,2)$, neither (H1) nor (H2) is 
known to be true and below we give some 
considerations related to this.  

\begin{Rk}\label{commentone}
The map $K(\eta \scriptstyle{\nearrow} \mspace{1.5mu}\displaystyle{}0)$ is an equivalence, 
and to show that (H1) holds, it suffices to show that for a cofinal subsequence $i_0, i_1, ..., i_l, ...$ of $\{i \geq 0\}$, each map $K(\eta \scriptstyle{\nearrow} \mspace{1.5mu}\displaystyle{}i_l)$ is an 
equivalence. For each $i$, the canonical map $E_n^{hG_n} \to E_n^{hU_i}$ is a $K(n)$-local $G_n/U_i$-Galois 
extension \cite[Theorem 5.4.4, (c)]{Rognes}, so that by, for example, (\ref{chromaticmorava}), this map is also a $T(n)$-local $G_n/U_i$-Galois extension, as noted in \cite[Section 4.3]{CMNN}. By 
[op. cit., Corollary 4.16], if $G_n/U_i$ is a $p$-group, then $K(\eta \scriptstyle{\nearrow} \mspace{1.5mu}\displaystyle{}i)$ is an equivalence. 
When $(n,p) = (1,2)$, $G_1 \cong \mathbb{Z}_2^\times \cong \mathbb{Z}_2 \times 
\mathbb{Z}/(2)$ is a pro-$2$-group, with $\mathbb{Z}_2$ equal to the $2$-adic integers, 
and hence, (H1) holds. 
\end{Rk} 

It is a special case of \cite[Conjecture 4.2]{AusoniRognesGuido}, due to Ausoni and Rognes, that for all $n$, $p$, and $i$, the canonical map 
\[L_{T(n+1)}K(E_n^{hG_n}) \to L_{T(n+1)}(K(E_n^{hU_i})^{hG_n/U_i})\] is an equivalence. 
Though ``Conjecture 4.2" is in general still open, this conjecture, results in \cite[Sections 1, 4]{CMNN}, especially [op. cit., Corollary 4.16] -- which was 
used in Remark \ref{commentone}, and \cite[Theorems 1.3, 1.8, 1.10, 5.1, 5.6]{ClausenEtAl}, which include verifying special cases of 
``Conjecture 4.2," give momentum for perhaps 
validating (H1) in every case. 

To underline the plausibility of (H1) in general, we briefly highlight 
\cite[Example 4.17]{CMNN} from the progress cited above. Let 
$E$ be a Lubin-Tate theory with extended Morava stabilizer group 
$G_n^\mathrm{ex} \cong {S}_n \rtimes \widehat{\mathbb{Z}}$, where $S_n$ is the Morava stabilizer group and $\widehat{\mathbb{Z}}$ is the 
profinite completion of the integers, and for $K$ a closed subgroup of $G_n^\mathrm{ex}$, let 
$E^{hK}$ denote the continuous homotopy fixed points. As explained in \cite[Sections 5.1, 5.2]{bldgs}, the construction of $E^{hK}$ uses \cite{DH}. Now let ${U}$ be an open subgroup of 
$G_n^\mathrm{ex}$ such that $U \cap S_n$ is pro-$p$. 
By \cite[Corollary 4.16]{CMNN} and \cite{ClausenEtAl}, given any 
normal inclusion $V' \vartriangleleft V \subset U$ of open subgroups, the canonical map
\[L_{T(n+1)}K(E^{hV}) \xrightarrow{\,\simeq\,} (L_{T(n+1)}K(E^{hV'}))^{hV/V'}\] is an equivalence, and this yields a sheaf of $T(n+1)$-local spectra on the site $\mathcal{T}_{{U}}$. 

As explained in Definition \ref{iv}, (H2) holds if for every $j \geq 0$, the presheaf $\mathcal{P}_j$ of Section \ref{oneone} satisfies ``condition (iv)" 
with $\mathcal{N} = \{U_i \mid i \geq 0\}$ (``condition (iv)" is based on 
\cite{ClausenMathewInventiones}). Related 
to this is the familiar problem of showing that a filtered colimit of homotopy spectral sequences 
has abutment equal to the colimit of the abutments of those spectral sequences (for 
example, see \cite[Section 3.1.3]{Mitchell}). Given any $j$, for each $i$ 
there is the homotopy fixed point spectral 
sequence $\{\,{^i_{\mspace{-3mu}j}}\mspace{-2mu}E_r^{\ast,\ast}\}_{r \geq 1}$ that has 
the form
\[{^i_{\mspace{-3mu}j}}\mspace{-2mu}E_2^{s,t} \mspace{-2mu} = \mspace{-1mu} H^s(G_n/U_i, \pi_t(K(E_n^{hU_i}) \otimes L_{n+1}^fM_j)) 
\Rightarrow \pi_{t-s}\bigl((K(E_n^{hU_i}) \otimes L_{n+1}^fM_j)^{hG_n/U_i}\bigr).\] Then (H2) is valid if for each $j$, there is some $r \geq 2$ and some integer 
$s'$ such that
\[{^i_{\mspace{-3mu}j}}\mspace{-2mu}E_r^{s,t} = 0, \ \text{for all} \ i \geq 0, \, s > s', \, t \in \mathbb{Z}.\]

\subsection{Possible connections with the Ausoni-Rognes conjecture without using towers} 
The next result is an immediate consequence of the following 
definition (and, for example, \cite[Theorem 3.2.1]{joint} and \cite[Theorem 7.9]{cts}).

\begin{Def}\label{defwithouttower}
Let $n \geq 1$ and set $p$ equal to any prime. Recall that $T(n+1)$ denotes any choice of 
a representative from the Bousfield class of $v_{n+1}^{-1}F(n+1)$, where $F(n+1)$ is any finite type $n+1$ complex.
Since
\[K(E_n) \otimes T(n+1)
\simeq \colim_{i \geq 0} (K(E_n^{hU_i})\otimes T(n+1)),\] where for each $i$, the copy of 
$T(n+1)$ is equipped with the trivial $G_n/U_i$-action, it is natural to define the 
continuous homotopy fixed point spectrum
\[(K(E_n) \otimes T(n+1))^{hG_n} := \bigl(\colim_{i \geq 0} (K(E_n^{hU_i})\otimes T(n+1))\bigr)^{hG_n},\] 
which, as in Definition \ref{fun}, is the global sections of a Postnikov sheafification. 
\end{Def}

\begin{Thm}\label{notower}
For each $n \geq 1$ and a prime $p$, there is a conditionally convergent 
homotopy fixed point spectral sequence
\[E_2^{s,t} = H^s_c(G_n, \pi_t(K(E_n) \otimes T(n+1))) \Longrightarrow 
\pi_{t-s}((K(E_n) \otimes T(n+1))^{hG_n}),\] where 
$\pi_t(K(E_n) \otimes T(n+1))$ is a discrete $G_n$-module, for each $t \in \mathbb{Z}$.
\end{Thm} 

By 
\cite[page 46; Remark 10.8]{jems}, the Ausoni-Rognes conjecture 
predicts that there is a homotopy fixed point spectral sequence
\begin{equation}\label{conjecturedhfps}\zig
H^s_c(G_n, \pi_t(K(E_n) \otimes v_{n+1}^{-1}F(n+1))) \Longrightarrow 
\pi_{t-s}(K(L_{K(n)}S^0) \otimes v_{n+1}^{-1}F(n+1)),\end{equation} and thus, 
it is natural to ask the following questions:
\begin{itemize}
\item
When $T(n+1) := v_{n+1}^{-1}F(n+1)$, the spectral sequence in Theorem \ref{notower} and 
conjectural spectral sequence (\ref{conjecturedhfps}) have identical 
$E_2$-terms. Is the former spectral sequence a 
realization of the latter one? 
\item 
In general, is there an equivalence between 
\[K(L_{K(n)}S^0) \otimes T(n+1) \ \text{and} \ (K(E_n) \otimes T(n+1))^{hG_n}?\] 
\item 
What is the relationship between 
\[(K(E_n) \otimes T(n+1))^{hG_n} \ \text{and} \ K(E_n)^{hG_n} \otimes T(n+1)?\] 
\end{itemize} 
The following result gives two conditions that imply cases in which the answers 
to the first and second questions above are ``yes." 

\begin{Thm}\label{moreprogressmaybe}
Let $n \geq 1$, let $p$ be a prime, and set $T(n+1)$ equal to $v_{n+1}^{-1}F(n+1)$, where 
$F(n+1)$ is an atomic $\mu$-spectrum. If \begin{itemize} 
\item[{$\mathrm{(H1}^{\prime}\mathrm{)}$}] the map $\displaystyle{\colim_{i \geq 0} K(\eta} \scriptstyle{\nearrow} \mspace{1.5mu}\displaystyle{}i)$ is a $T(n+1)$-equivalence, 
and 
\item[{$\mathrm{(H2}^{\prime}\mathrm{)}$}]
the canonical map $\displaystyle{\colim_{i \geq 0} \mathrm{Tot}(C(n, i)^\bullet)
\to \mathrm{Tot}(\colim_{i \geq 0} C(n,i)^\bullet)}$ is an equivalence, 
where for each $i$, $C(n,i)^\bullet := \textstyle{\prod}_{(G_n/U_i)^\bullet} (K(E_n^{hU_i}) \otimes T(n+1))$,
\end{itemize} 
then there is an equivalence
\[(K(E_n) \otimes T(n+1))^{hG_n} \simeq K(L_{K(n)}S^0) \otimes T(n+1).\] 
\end{Thm}

The proof of Theorem \ref{moreprogressmaybe} is in Section \ref{strategy}. When $(n,p) = 
(1,2)$, assumption $\mathrm{(H1}^{\prime}\mathrm{)}$ is true, by Remark \ref{commentone}, but 
for all other pairs $(n,p)$, neither {$\mathrm{(H1}^{\prime}\mathrm{)}$} nor {$\mathrm{(H2}^{\prime}\mathrm{)}$} is known to hold. 

The next result, whose proof is in Section \ref{sectiondiscrete}, has as a consequence that 
if $F(n+1)$ is chosen to be an atomic $\mu$-spectrum, then (by Remark \ref{ring2}) with 
$T(n+1)$ set equal to $v_{n+1}^{-1}F(n+1)$, 
$(K(E_n) \otimes T(n+1))^{hG_n}$ is $T(n+1)$-local. 

\begin{Thm}\label{discretetwo}
When $n \geq 1$, $p$ is a prime, and $T(n+1) := v_{n+1}^{-1}F(n+1)$, where $F(n+1)$ is 
any finite type $n+1$ complex, there are equivalences 
\begin{align*}
(K(E_n) \otimes T(n+1))^{hG_n} & \simeq 
(\colim_{i \geq 0} L_{n+1}^f K(E_n^{hU_i}))^{hG_n}\otimes F(n+1)\\
& \simeq (\colim_{i \geq 0} L_{n+1}^f K(E_n^{hU_i}))^{hG_n}\otimes T(n+1)\\
& \simeq (L_{T(n+1)}K(E_n))^{hG_n} \otimes T(n+1).\end{align*}
\end{Thm} 

Let $n=1$, $p \geq 5$, and let $V(1)$ denote the type $2$ Smith-Toda 
complex $S^0/(p,v_1)$. In \cite{padicspectra2}, we 
constructed in the setting of symmetric spectra 
of simplicial sets the continuous homotopy fixed 
points $(K(E_1) \wedge v_2^{-1}V(1))^{h\mathbb{Z}_p^\times}$. In symmetric spectra, there should be a zigzag of weak equivalences 
between this $(K(E_1) \wedge v_2^{-1}V(1))^{h\mathbb{Z}_p^\times}$ and 
$\bigl(\colim_{i \geq 0} (K(E_1^{hU_i}) \wedge v_2^{-1}V(1))\bigr)^{hG_1}$ -- the model for the 
continuous homotopy fixed points of $K(E_1) \wedge v_2^{-1}V(1)$ 
constructed in Definition \ref{defwithouttower}, but we have not completed our work on this zigzag. 
In \cite{padicspectra2}, we also obtained a homotopy fixed point spectral sequence for 
$(K(E_1) \wedge v_2^{-1}V(1))^{h\mathbb{Z}_p^\times}$ with $E_2$-term equal to the 
$E_2$-term of the spectral sequence given by 
Theorem \ref{notower} (with $T(2)$ there set equal to $v_2^{-1}V(1)$) and these two spectral sequences 
should be isomorphic, but our work on this is incomplete, since it is closely related to 
the aforementioned zigzag. 

\subsection*{Acknowledgements} 
I thank Martin Frankland for a discussion about $L_{T(n)}$, Niko Naumann for introducing me to a version of ``condition (iv)" in Definition \ref{iv}, Akhil Mathew and 
John Rognes for helpful interactions, and Philip Hackney and 
Justin Lynd for stimulating conversations related to working with $\infty$-categories. 

\section{Some basic facts about $T(n)$-localization}
As in the introduction, $n \geq 1$, $p$ is any prime, and $F(n)$ is a 
finite type $n$ complex. Also, following \cite[Proposition 4.22]{HS}, we let
\[M_0 \leftarrow M_{1} \leftarrow \cdots \leftarrow M_{j} \leftarrow \cdots\] be a tower of generalized 
Moore spectra, with each $M_{j}$ finite of type $n$ (here, we have type $n$, not type $n+1$, as in the introduction) and an atomic $\mu$-spectrum. 
As recalled in Section \ref{oneone}, one feature of this tower is that 
for any $Z \in \mathrm{Sp}$,
\[L_{F(n)} Z \simeq \lim_{j \geq 0} (Z \otimes M_{j}).\] 
Given a finite type $0$ 
complex $F(0)$, we let $T(0) := v_0^{-1}F(0)$ denote the telescope of a $v_0$-self-map on 
$F(0)$. Thus, $T(0)$ and $H\mathbb{Q}$ have the same Bousfield class.  

\begin{Rk}\label{ring}
Suppose that $F(n)$ is an atomic $\mu$-spectrum. Then each of $F(n)$ and, by \cite[proof of Lemma 2.2]{Mahowald/Sadofsky}, the telescope 
$v_n^{-1}F(n)$ is a ``ring spectrum," in the sense of 
\cite{Mahowald/Sadofsky} (see also \cite{Devinatzrings}). Here, by ``ring spectrum," we mean 
a left-unital magma in the homotopy category of spectra (that need not be associative or right-unital). It follows that $T(n)$ can be taken to 
be a ``ring spectrum" in the above sense, and thus, it is worth highlighting the fact that \cite[proof of Lemma 2.3]{LMMT} proves the much stronger result that $T(n)$ can be set equal to an algebra in $\mathrm{Sp}$. 
\end{Rk}

\begin{Rk}\label{ring2}
Again, let $F(n)$ be an atomic $\mu$-spectrum, so that as in Remark \ref{ring}, $v_n^{-1}F(n)$ is a ``ring spectrum." Then if $Z$ is any spectrum, $Z \otimes v_n^{-1}F(n)$ is $T(n)$-local. To verify this, it 
suffices to show that the equivalent spectrum $v_n^{-1}F(n) \otimes Z$ is $v_n^{-1}F(n)$-local 
(since $v_n^{-1}F(n)$ and $T(n)$ have the same Bousfield class), and this conclusion is reached by noting that the argument in \cite[proof of Proposition 1.17, (a)]{AJMperiodic} goes through here. The observation that this argument applies in this context also occurs in the antepenultimate paragraph of \cite[proof of Lemma 2.2]{Mahowald/Sadofsky}. 
\end{Rk} 

We return to letting $F(n)$ denote a finite type $n$ complex that is not necessarily an atomic $\mu$-spectrum. 
We recall some standard notation (for example, see \cite[Definition 3.1]{Mahowald/Sadofsky}). 

\begin{Def}\label{finitelocalize}
For each $n \geq 1$ and every prime $p$, $L_n^f$ denotes the Bousfield localization functor determined by the spectrum $T(0) \oplus T(1) \oplus \cdots \oplus T(n)$. 
By \cite[Corollary 3.5]{Mahowald/Sadofsky}, $L_n^f$ is smashing. 
\end{Def} 

\begin{Rk}\label{telescope}
By \cite[Proposition 3.2]{Mahowald/Sadofsky}, there is an equivalence
\[L_n^f F(n) \simeq v_n^{-1} F(n).\]
\end{Rk}

We believe the following result is fairly well-known (for example, see \cite[3.2; Theorem 3.3]{Bousfieldtelescopic} and \cite[Corollary 2.2]{HoveyCech}), but we do not know of a 
reference to it in the literature that -- relative to the setup and definitions in this paper -- 
is straightforward to follow, and so we give a proof. We learned of this result from \cite[Fact 2.11, 2]{Franklandnotes} and in the case when $Z = S^0$, the result is 
\cite[Proposition 5.1]{Mahowald/Sadofsky}.

\begin{Thm}\label{split}
Given $Z \in \mathrm{Sp}$, $n \geq 1$, and $p$ any prime, there is an equivalence 
\[L_{T(n)}Z \simeq L_{F(n)}L_n^f Z.\] 
\end{Thm}

\begin{proof} 
Since $L_n^f$ is smashing and, as in \cite[Proposition 5.1]{Mahowald/Sadofsky}, there 
is the diagram $\{v_n^{-1}M_j\}_{j \geq 0}$ of telescopes, there are equivalences
\[L_{F(n)}L_n^fZ \simeq \lim_{j \geq 0} (Z \otimes L_n^f M_j) \simeq \lim_{j \geq 0} (Z \otimes v_n^{-1} M_j).\]
Each $M_j$ is an atomic $\mu$-spectrum, so that each $Z \otimes v_{n}^{-1}M_j$ is 
$T(n)$-local, and hence, the three displayed expressions above are $T(n)$-local. Now we only need to show 
that the composition \[Z \to 
L_n^fZ \to L_{F(n)}L_n^fZ\] of canonical maps is a $T(n)$-equivalence. 

We fix a choice for $T(n)$: let \[T(n) = v_{n}^{-1}M_0= \colim_{k \geq 0} \Sigma^{-kd} M_0,\] where $d$ is a fixed integer determined by the self-map used to form the telescope. As in Remark \ref{ring2}, since 
$M_0$ is an atomic $\mu$-spectrum, the argument in \cite[proof of Proposition 1.17, (a)]{AJMperiodic} shows that $L_n^fS^0 \otimes Z \otimes \Sigma^{-kd}M_0$ is $M_0$-local, and hence, $F(n)$-local, since $M_0$ and $F(n)$ have the same Bousfield class (by \cite[page 5]{nilpotencetwo}), for all $k \geq 0$. This justifies the third step in the 
following chain of equivalences, whose first step applies the fact that $\Sigma^{-kd} M_0$ is a 
finite spectrum:
\begin{align*} 
(L_{F(n)}& L_n^fZ) \otimes T(n) \simeq \colim_{k \geq 0} \lim_{j \geq 0} (L_n^fS^0 \otimes Z \otimes M_j \otimes \Sigma^{-kd}M_0)\\
& \simeq \colim_{k \geq 0} L_{F(n)}(L_n^fS^0 \otimes Z \otimes \Sigma^{-kd}M_0) \simeq \colim_{k \geq 0} 
(L_n^fS^0 \otimes Z \otimes \Sigma^{-kd}M_0)\\ & \simeq 
L_n^fS^0 \otimes Z \otimes T(n)
\simeq Z \otimes L_n^fT(n) \simeq Z \otimes L_n^f(L_n^fM_0) 
\simeq Z \otimes L_n^fM_0\\ & \simeq Z \otimes T(n).\end{align*} It follows that 
the aforementioned composition is a $T(n)$-equivalence. 
\end{proof}

\begin{Rk}\label{canusejoint}
In \cite[Section 4.3]{CMNN}, the authors work with $T(n)$-local and $T(n)$-local pro-Galois extensions in the sense of \cite{Rognes}. Theorem \ref{split} above shows that $T(n)$ satisfies \cite[Assumption 1.0.3]{joint} and so \cite{joint} can be used to study ``$T(n)$-local profinite Galois extensions" (especially ones that are consistent and of finite virtual cohomological dimension), 
which differ slightly from $T(n)$-local pro-Galois extensions. 
\end{Rk}

As in \cite[Section 3]{finitebousfield}, set 
\[K({\leq n}) := K(0) \oplus K(1) \oplus \cdots \oplus K(n),\] where 
$K(0) = H\mathbb{Q}$, and as is standard, we let $L_n$ 
denote the Bousfield localization functor $L_{K({\leq n})}$. Notice that given a $p$-local 
spectrum $Z$, 
the canonical $K({\leq n})$-equivalence 
$L_n^f Z \to L_nZ$ (see \cite[page 113]{Mahowald/Sadofsky}) is an equivalence when 
$Z$ is $K({\leq n})$-local, since there are equivalences
\[L_n^f Z \simeq L_n^f S^0 \otimes Z \simeq L_nS^0 \otimes L_n^f S^0 \otimes Z \simeq L_nL_n^f Z \xrightarrow{\,\simeq\,} L_nL_nZ \simeq Z.\] Therefore, if a spectrum $Z$ is 
$K({\leq n})$-local, there are equivalences
\begin{equation}\label{chromaticmorava}\zig
L_{K(n)}Z \simeq L_{F(n)}L_n Z \xleftarrow{\,\simeq\,} L_{F(n)}L_n^f Z \simeq L_{T(n)}Z,
\end{equation}
where the first and last steps applied \cite[Proposition 7.10, (e)]{HS} and Theorem \ref{split}, 
respectively. The observation in (\ref{chromaticmorava}) 
is not original: it is stated in \cite[Section 3]{Ohkawa} and \cite[proof of Corollary 4.20, (iv)]{LMMT}, and 
the latter reference gives a proof. 

\section{A proof of Theorem \ref{gettingtheballrolling}}\label{proofofrollingball} 

We continue to let $n \geq 1$ and $p$ denotes a prime. 
Let \[A_0 \to A_1 \to \cdots \to A_i \to A_{i+1} \to \cdots\] be a diagram in the $\infty$-category $\mathcal{A}_{T(n)}$ of $T(n)$-local ring spectra. 
The colimit $\colim_{i \geq 0}^{\mathcal{A}_{T(n)}} A_i$ of this diagram in $\mathcal{A}_{T(n)}$ satisfies
\[\textstyle{\colim_{i \geq 0}^{\mathcal{A}_{T(n)}} A_i} \simeq \displaystyle{L_{T(n)}(\colim_{i \geq 0} A_i)}\] in $\mathcal{A}_{T(n)}$. 
Then a special case of \cite[Corollary 4.31]{LMMT}, which is 
due to Land, Mathew, Meier, and Tamme, is the remarkable fact that there is an equivalence 
\[L_{T(n+1)}K(\textstyle{\colim_{i \geq 0}^{\mathcal{A}_{T(n)}} A_i}) \simeq \displaystyle{L_{T(n+1)}(\colim_{i \geq 0} L_{T(n+1)}K(A_i))}\] in $\mathrm{Sp}$, which simplifies to 
\begin{equation}\label{neatLMMT}\zig
L_{T(n+1)}K(L_{T(n)}(\displaystyle{\colim_{i \geq 0} A_i})) \simeq \displaystyle{L_{T(n+1)}(\colim_{i \geq 0} K(A_i))}\end{equation} in $\mathrm{Sp}$.

Now we recall that (\ref{diagram}) is the diagram 
$E_n^{hG_n} = E_n^{hU_0} \to E_n^{hU_1} \to \cdots$ 
in the category $\mathrm{CAlg}(\mathrm{Sp})$ of 
commutative algebras, and 
since the forgetful functor $\mathrm{CAlg}(\mathrm{Sp}) \to \mathrm{Sp}$ detects the colimit in $\mathrm{CAlg}(\mathrm{Sp})$ for such a diagram \cite[Corollary 3.2.3.2]{LurieHA}, there is an equivalence
\[E_n \simeq L_{K(n)}(\colim_{i \geq 0} E_n^{hU_i})\] (see \cite[Definition 1.5, Lemma 6.2 and its proof, Proposition 6.4]{DH}). 

For each $i$, $E_n^{hU_i}$ is $K(n)$-local, so that $E_n^{hU_i}$ is $T(n)$-local 
and $\colim_{i \geq 0} E_n^{hU_i}$ is 
$E_n$-local. Then by (\ref{chromaticmorava}), 
\[L_{K(n)}(\colim_{i \geq 0} E_n^{hU_i}) \simeq L_{T(n)}(\colim_{i \geq 0} E_n^{hU_i}).\] 
Furthermore, (\ref{diagram}) is a diagram 
in $\mathcal{A}_{T(n)}$, and hence, equivalence (\ref{neatLMMT}) yields that
\[L_{T(n+1)}K(E_n) \simeq L_{T(n+1)}K(L_{T(n)}(\colim_{i \geq 0} E_n^{hU_i})) \simeq L_{T(n+1)}(\colim_{i \geq 0} K(E_n^{hU_i})).\] 
We recall the tower  
\[M_0 \leftarrow M_1 \leftarrow \cdots \leftarrow M_j \leftarrow \cdots\] 
from Section \ref{oneone} of type $n+1$ generalized Moore spectra. 
Then by Theorem \ref{split}, there is an equivalence
\[L_{T(n+1)}K(E_n) \simeq 
\lim_{j \geq 0} (\colim_{i \geq 0} (K(E_n^{hU_i}) \otimes L_{n+1}^f M_j)),\] 
which completes the proof of Theorem \ref{gettingtheballrolling}. 

\section{Continuous homotopy fixed points and Postnikov sheafification}\label{ctshfps}
In this section, we briefly recall some background material on continuous homotopy fixed point spectra and we make some observations about relationships with (pre)sheaves of spectra.

Let $G$ be any profinite group. Let 
$\mathcal{N}$ be a collection of open normal subgroups of $G$ that is cofinal in the collection of all the open normal subgroups of $G$. Suppose that 
$\{X_N\}_{N \in \mathcal{N}}$ is a diagram of $G$-spectra, consisting of a single map $X_N \to X_{N'}$ for each inclusion $N' \subset N$ in $\mathcal{N}$, such that for each $N \in \mathcal{N}$, the $G$-action on $X_N$ factors through $G/N$. Then $\colim_{N \in \mathcal{N}} X_N$ has an induced $G$-action, there is the continuous homotopy fixed point spectrum $(\colim_{N \in \mathcal{N}} X_N)^{hG}$, and if one of the conditions
\begin{itemize}
\item[{(i)}]
$G$ has finite virtual cohomological dimension;
\item[{(ii)}] 
there is a fixed integer $m$ such that $H^s_c(N', \pi_t(\colim_{N \in \mathcal{N}} X_N)) = 0$, for 
all $s > m$, $t \in \mathbb{Z}$, and $N' \in \mathcal{N}$; and
\item[{(iii)}] 
there is a fixed integer $r$ such that $\pi_t(\colim_{N \in \mathcal{N}} X_N) = 0$, for all $t >r$
\end{itemize} holds, then by \cite[Theorem 3.2.1; page 5038: 2nd paragraph]{joint} and \cite[page 911]{2ndnyjm}, there is an equivalence
\begin{equation}\label{model}\zig
(\colim_{N \in \mathcal{N}} X_N)^{hG} \simeq \mathrm{Tot}(\colim_{N \in \mathcal{N}} (\textstyle{\prod}_{(G/N)^\bullet} X_N)),
\end{equation}
where for each $N$, $\textstyle{\prod}_{(G/N)^\bullet} X_N$ is a cosimplicial spectrum 
that satisfies
\[\mathrm{Tot}(\textstyle{\prod}_{(G/N)^\bullet} X_N) \simeq (X_N)^{hG/N},\] with 
$(\prod_{(G/N)^\bullet} X_N)([0]) = \prod_{(G/N)^0} X_N = X_N$ and for each $n \geq 1$, 
\[(\textstyle{\prod_{(G/N)^\bullet} X_N})([n]) = \textstyle{\prod}_{(G/N)^n} X_N\] is the product of copies of $X_N$ indexed by the 
set $(G/N)^n$, which is the $n$-fold product of copies of $G/N$. 
As is well-known, the cosimplicial spectrum $\textstyle{\prod}_{(G/N)^\bullet} X_N$ 
that can be used in this discussion is not unique. 

In general (that is, even when none of conditions (i) -- (iii) are satisfied), associated to $\colim_{N \in \mathcal{N}} X_N$ is the presheaf $\mathcal{P} \in \mathrm{PSh}(\mathcal{O}_G,\mathrm{Sp})$ defined by 
\[G/U \mapsto \mathcal{P}(G/U) := (\colim_{N \in \mathcal{N}} X_N)^{hU}, \ \ \ U \ \text{an open subgroup of }G,\] 
where $(\colim_{N \in \mathcal{N}} X_N)^{hU}$ is the continuous $U$-homotopy fixed points (for example, see \cite[Proposition 3.3.1]{joint} and \cite[Section 2]{tmfbook}). This presheaf extends canonically to a presheaf $\mathcal{P}_{\scriptscriptstyle{\prod}}$ in $\mathrm{PSh}_{\scriptscriptstyle{\prod}}(\mathcal{T}_G, \mathrm{Sp})$, by sending finite coproducts in $\mathcal{T}_G$ to finite products in $\mathrm{Sp}$.  

In the other direction, let $\mathcal{F}$ be a presheaf in $\mathrm{PSh}(\mathcal{O}_G,\mathrm{Sp})$, so that for each $N \in \mathcal{N}$, $\mathcal{F}(G/N)$ has a natural $G$-action that factors through the $G/N$-action, and, as usual, the latter action yields the cosimplicial spectrum 
$\prod_{(G/N)^\bullet} \mathcal{F}(G/N)$. Let 
$\widetilde{\mathcal{F}}_{\scriptscriptstyle{\prod}} \in \mathrm{Sh}(\mathcal{T}_G, \mathrm{Sp})$ denote the Postnikov sheafification of the canonical presheaf $\mathcal{F}_{\scriptscriptstyle{\prod}}$ in $\mathrm{PSh}_{\scriptscriptstyle{\prod}}(\mathcal{T}_G, \mathrm{Sp})$ that is induced by $\mathcal{F}$. Then by 
\cite[Construction 4.6, proof of Proposition 4.9]{ClausenMathewInventiones},
\[\widetilde{\mathcal{F}}_{\scriptscriptstyle{\prod}}(\ast) \simeq \mathrm{Tot}(\colim_{N \in \mathcal{N}} \mathcal{F}_{\scriptscriptstyle{\prod}}((G/N)^{\bullet+1})),\] where for each $N$, 
$(G/N)^{\bullet+1}$ is the usual simplicial object in $\mathcal{T}_{G}$ associated to $G/N$. 
For each $n \geq 0$, the stabilizer subgroup in $G$ of the $G$-action on any element in $(G/N)^{n+1}$ is 
$N$, so that there is an isomorphism $(G/N)^{n+1} \cong \coprod_{(G/N)^n} G/N$ in $\mathcal{T}_G$. It follows that there is an equivalence
\[\widetilde{\mathcal{F}}_{\scriptscriptstyle{\prod}}(\ast) \simeq \mathrm{Tot}(\colim_{N \in \mathcal{N}} \textstyle{\prod}_{(G/N)^\bullet} \mathcal{F}(G/N))\] (for example, see 
\cite[proof of Proposition 4.9, (7) in Proposition 4.11]{ClausenMathewInventiones}).

Each inclusion $N' \subset N$ in $\mathcal{N}$ induces the projection $\pi \: G/N' \to G/N$ in $\mathcal{O}_G$ and 
the map $\mathcal{F}(\pi)\: \mathcal{F}(G/N) \to \mathcal{F}(G/N')$, with source and target equipped with the induced $G$-action, is $G$-equivariant, so that as at the beginning of this section, $\colim_{N \in \mathcal{N}} \mathcal{F}(G/N)$ has a $G$-action. 

The following result is an immediate consequence of the above. 

\begin{Thm}\label{sheafone}
Let $G$ be a profinite group and let $\mathcal{F} \in \mathrm{PSh}(\mathcal{O}_G, \mathrm{Sp})$, with $\mathcal{F}(G/N)$ 
equipped with the natural $G/N$-action for each $N \in \mathcal{N}$. If $G$ and 
$\colim_{N \in \mathcal{N}} \mathcal{F}(G/N)$ satisfy one of conditions $\mathrm{(i)}\,$--$\,\mathrm{(iii)}$, then 
\[(\colim_{N \in \mathcal{N}} \mathcal{F}(G/N))^{hG} \simeq \widetilde{\mathcal{F}}_{\scriptscriptstyle{\prod}}(\ast),\] where the right-hand side of this equivalence is the global sections of the Postnikov sheafification of ${\mathcal{F}}_{\scriptscriptstyle{\prod}}$. 
\end{Thm}

\begin{Def}\label{iv}
Given a profinite group $G$ and $\mathcal{F} \in \mathrm{PSh}(\mathcal{O}_G, \mathrm{Sp})$, 
we define ``condition (iv)" to be
\begin{itemize}
\item[{(iv)}] 
there is some integer $d \geq 0$ such that for each $N \in \mathcal{N}$, 
the $G/N$-action on $\mathcal{F}(G/N)$ is weakly $d$-nilpotent.
\end{itemize}
We refer the reader to \cite[Definition 4.8]{ClausenMathewInventiones} for the meaning of  
``weakly $d$-nilpotent." Our consideration of condition (iv) is partly motivated by [op. cit., Propositions 
4.9, 4.16], and [op. cit., Proposition 4.16, Theorem 4.26 (see its proof and the paragraph 
above Theorem 4.25)] give scenarios implying that condition (iv) holds. When this 
condition is satisfied, [op. cit., Lemma 2.34] gives 
\[\widetilde{\mathcal{F}}_{\scriptscriptstyle{\prod}}(\ast) \simeq \colim_{N \in \mathcal{N}} 
\mathrm{Tot}(\textstyle{\prod}_{(G/N)^\bullet} \mathcal{F}(G/N)) \simeq 
\displaystyle{\colim_{N \in \mathcal{N}} \mathcal{F}(G/N)^{hG/N}}.\]  
\end{Def} 

Now we put together the various strands of discussion of this section in the following result. 

\begin{Thm}
Let $G$ be a profinite group and suppose that $\{X_N\}_{N \in \mathcal{N}}$ is a diagram of 
$G$-spectra consisting of a unique map $X_N \to X_{N'}$ whenever $N' \subset N$ in $\mathcal{N}$, such that for each $N \in \mathcal{N}$, the $G$-action on $X_N$ factors through $G/N$. Let $\widetilde{\mathcal{F}}_{\scriptscriptstyle{\prod}}$ be the Postnikov sheafification 
of the presheaf $\mathcal{F}_{\scriptscriptstyle{\prod}} \in \mathrm{PSh}_{\scriptscriptstyle{\prod}}(\mathcal{T}_G,\mathrm{Sp})$ 
determined by the presheaf
\[\mathcal{F} \: G/U \mapsto (\colim_{N \in \mathcal{N}} X_N)^{hU}, \ \ \ \text{$U$ an open subgroup of $G$}.\] If $G$, $\colim_{N \in \mathcal{N}} X_N$, and $\mathcal{F}$ satisfy any one of conditions $\mathrm{(i)}\,$--$\,\mathrm{(iv)}$, then there are equivalences
\[(\colim_{N \in \mathcal{N}} X_N)^{hG} \simeq \widetilde{\mathcal{F}}_{\scriptscriptstyle{\prod}}(\ast) \simeq \mathrm{Tot}(\colim_{N \in \mathcal{N}} \textstyle{\prod}_{(G/N)^\bullet} X_N).\] 
\end{Thm}

\begin{proof}
In general (that is, even when none of conditions (i) -- (iv) hold), 
for each $n \geq 0$, there are equivalences
\begin{align*}
\colim_{N \in \mathcal{N}} \textstyle{\prod}_{(G/N)^n} \mathcal{F}(G/N) 
& \simeq \colim_{N \in \mathcal{N}} \textstyle{\prod}_{(G/N)^n} \displaystyle{\bigl(\colim_{N' \in \mathcal{N}} \mathcal{F}(G/N')\bigr)}\\
& = \colim_{N \in \mathcal{N}} \textstyle{\prod}_{(G/N)^n} \displaystyle{\bigl(\colim_{N' \in \mathcal{N}} \bigl(\colim_{N'' \in \mathcal{N}} X_{N''})^{hN'}\bigr)} 
\\ & \simeq \colim_{N \in \mathcal{N}} \textstyle{\prod}_{(G/N)^n} \displaystyle{(\colim_{N'' \in \mathcal{N}} X_{N''})} 
\simeq \colim_{N \in \mathcal{N}} \textstyle{\prod}_{(G/N)^n} X_N,
\end{align*}
where the penultimate step uses \cite[Proposition 3.3.1:~(2),~(3)]{joint}, and hence (again, 
in general), 
there are equivalences
\[
\widetilde{\mathcal{F}}_{\scriptscriptstyle{\prod}}(\ast) \simeq \mathrm{Tot}(\colim_{N \in \mathcal{N}} \textstyle{\prod}_{(G/N)^\bullet} \mathcal{F}(G/N)) 
\simeq \mathrm{Tot}(\displaystyle{\colim_{N \in \mathcal{N}}} \, \textstyle{\prod}_{(G/N)^\bullet} X_N).\]
These last two equivalences, together with the discussion at the beginning of this section, imply the conclusion of the theorem when (i), (ii), or (iii) is satisfied. 
 
Now suppose that condition (iv) holds: the desired conclusion comes from our last two equivalences and
\[\widetilde{\mathcal{F}}_{\scriptscriptstyle{\prod}}(\ast) \simeq \colim_{N \in \mathcal{N}} \mathcal{F}(G/N)^{hG/N} =
\colim_{N \in \mathcal{N}} \bigl((\colim_{N' \in \mathcal{N}} X_{N'})^{hN}\bigr)^{\mspace{-2mu}hG/N} \simeq 
(\colim_{N \in \mathcal{N}} X_{N})^{hG},\] where the last step follows from the fact that 
for each $N$, 
\[\bigl((\colim_{N' \in \mathcal{N}} X_{N'})^{hN}\bigr)^{\mspace{-2mu}hG/N} \simeq (\colim_{N' \in \mathcal{N}} X_{N'})^{hG},\] by [op. cit., Proposition 3.3.1,~(4)].  
\end{proof}

\section{Proofs of Theorems \ref{spectralsequences}, \ref{discrete}, and \ref{discretetwo}}
\label{sectiondiscrete}

We let $n$ be any positive integer and $p$ a prime. Also, we let $\{M_j\}_{j \geq 0}$ be the 
tower of finite type $n+1$ spectra that is described in Section \ref{oneone}.

Now we prove Theorem \ref{spectralsequences}. 
Let $j \geq 0$: by Remark \ref{telescope}, $L_{n+1}^fM_j 
\simeq v_{n+1}^{-1}M_j$, so that 
\begin{align*} 
\colim_{i \geq 0} & (K(E_n^{hU_i}) \otimes L_{n+1}^f M_{j}) 
\simeq L_{T(n+1)}(\colim_{i \geq 0} K(E_n^{hU_i}))
\otimes v_{n+1}^{-1}M_j\\ & \simeq (L_{T(n+1)} K(E_n))
\otimes v_{n+1}^{-1}M_j \simeq K(E_n)
\otimes v_{n+1}^{-1}M_j.\end{align*} Also, by \cite[Proposition 5.1]{Mahowald/Sadofsky} and 
Theorem \ref{split}, there is a tower 
$\{v_{n+1}^{-1} M_j\}_{j \geq 0}$ and as a tower, it is levelwise equivalent to $\{L_{n+1}^f M_j\}_{j \geq 0}$. With these observations, Theorem \ref{spectralsequences} 
is an immediate consequence of \cite[Proposition 3.1.2]{Geisser}, \cite[Theorems 8.5, 8.8]{cts}, and \cite[Theorem 2.2]{Jannsen}.

To prove Theorem \ref{discrete}, we set ``$T$"  in the notation of \cite[Section 6.1]{joint} equal to 
$T(0) \oplus T(1) \oplus \cdots \oplus T(n+1)$. We have
\begin{align*}
(L&_{T(n+1)}K(E_n))^{hG_n} \simeq 
\lim_{j \geq 0} \bigl(\colim_{i \geq 0}(L_{n+1}^fK(E_n^{hU_i}) \otimes M_j)\bigr)^{hG_n}\\
& \mspace{-5mu} \simeq \lim_{j \geq 0} \bigl((\colim_{i \geq 0} L_{n+1}^fK(E_n^{hU_i})) \otimes M_j\bigr)^{hG_n}
\simeq 
\lim_{j \geq 0} \bigl((\colim_{i \geq 0} L_{n+1}^fK(E_n^{hU_i}))^{hG_n} \otimes M_j\bigr)\\ 
& \mspace{-5mu} \simeq \mspace{-1mu} L_{F(n+1)}((\colim_{i \geq 0} L_{n+1}^fK(E_n^{hU_i}))^{hG_n}) \mspace{-1mu} \simeq \mspace{-1mu}
L_{F(n+1)}L_{n+1}^f((\colim_{i \geq 0} L_{n+1}^fK(E_n^{hU_i}))^{hG_n})\\
& \mspace{-5mu} \simeq L_{T(n+1)}((\colim_{i \geq 0} L_{n+1}^f K(E_n^{hU_i}))^{hG_n}),
\end{align*} 
where the first step is because $L_{n+1}^f$ is smashing; the third step follows from the fact that 
each $M_j$ is a finite spectrum and $G_n$ has finite virtual cohomological dimension (for 
example, see \cite[Proposition 3.10]{Mitchell} and \cite[Theorem 3.2.1]{joint}); 
because $\displaystyle{\colim_{i \geq 0} L_{n+1}^fK(E_n^{hU_i})}$ is $T$-local (since $L_{n+1}^f$ is smashing), $\displaystyle{(\colim_{i \geq 0} L_{n+1}^fK(E_n^{hU_i}))^{hG_n}}$ 
is $T$-local, by \cite[proof of Lemma 6.1.5]{joint}, and this yields the fifth step; and 
Theorem \ref{split} 
gives the last step. This completes the proof of Theorem \ref{discrete}. 

We continue with the above context and prove Theorem \ref{discretetwo}: given any 
choice of $F(n+1)$, with $T(n+1)$ now set equal to $v_{n+1}^{-1}F(n+1)$, 
\begin{align*}
(K(E_n) & \otimes T(n+1))^{hG_n} 
\simeq \mathrm{Tot}\bigl(\colim_{i \geq 0} \textstyle{\prod}_{(G_n/U_i)^\bullet} 
(K(E_n^{hU_i}) \otimes L_{n+1}^fF(n+1))\bigr)
\\
& \simeq \bigl(\mathrm{Tot}(\colim_{i \geq 0} \textstyle{\prod}_{(G_n/U_i)^\bullet} 
L_{n+1}^f K(E_n^{hU_i}))\bigr) \otimes F(n+1)\\
& \simeq (\colim_{i \geq 0} L_{n+1}^f K(E_n^{hU_i}))^{hG_n}\otimes F(n+1)
\\ & \simeq 
L_{n+1}^f((\colim_{i \geq 0} L_{n+1}^f K(E_n^{hU_i}))^{hG_n})\otimes F(n+1)
\\ & \simeq 
(\colim_{i \geq 0} L_{n+1}^f K(E_n^{hU_i}))^{hG_n}\otimes T(n+1)\\ &
\simeq (L_{T(n+1)}K(E_n))^{hG_n} \otimes T(n+1),\end{align*} where the second equivalence applies the fact that $F(n+1)$ is a finite spectrum, the fourth equivalence uses the justification 
given above for the fifth step of the proof of Theorem \ref{discrete}, and the last equivalence is by Theorem \ref{discrete}. In this sequence of equivalent expressions, the first, fourth, sixth, and seventh ones are the ones explicitly required by Theorem \ref{discretetwo}. 

\section{How the two hypotheses give $(L_{T(n+1)}K(E_n))^{hG_n} \simeq L_{T(n+1)}K(L_{K(n)}S^0)$}\label{strategy}
In this section we prove Theorem \ref{wouldgiveprogress}. Also, since the proof is helpful for verifying the 
related Theorem \ref{moreprogressmaybe}, after giving the proof of the former result, we prove the latter one. 

Let $n \geq 1$, with $p$ equal to any prime. We let $\{M_j\}_{j \geq 0}$ be the tower of generalized Moore spectra from 
Section \ref{oneone}, with 
each $M_j$ a finite spectrum of type $n+1$ and an atomic $\mu$-spectrum, and $T(n+1)$ denotes a representative of the Bousfield class of $v_{n+1}^{-1}M_0$. We assume that the following two statements are 
true:
\begin{itemize}
\item[{(H1)}]
The map
\[\ \ \ \ \ \ \ \colim_{i \geq 0} K(\eta \scriptstyle{\nearrow} \mspace{1.5mu}\displaystyle{}i) \: 
L_{T(n+1)}K(E_n^{hG_n}) \to \colim_{i \geq 0} (L_{T(n+1)}K(E_n^{hU_i}))^{hG_n/U_i}\] 
is a $T(n+1)$-equivalence.
\item[{(H2)}]
For each $j \geq 0$, with 
$C(n,j,i)^\bullet := \textstyle{\prod}_{(G_n/U_i)^\bullet} (K(E_n^{hU_i}) \otimes L_{n+1}^fM_j)$ 
for all $i \geq 0,$ the canonical map
$\displaystyle{\colim_{i \geq 0} \displaystyle{} \mathrm{Tot}(C(n,j, i)^\bullet)
\to \mathrm{Tot}(\colim_{i \geq 0} C(n,j,i)^\bullet)}$ is an equivalence. 
\end{itemize} 
To prove Theorem \ref{wouldgiveprogress}, we use each assumption only once and the usages are marked with ``{\it By} (H1)" and 
``{\it applies} (H2)."

By Remark 
\ref{ring2}, for every $j$ and any spectrum $Z$, $Z \otimes v_{n+1}^{-1}M_j$ is $T(n+1)$-local, 
and hence, by Remark \ref{telescope}, $Z \otimes L_{n+1}^fM_j$ is $T(n+1)$-local. Also, for 
all $j$ and $Z$, 
\[Z \otimes L_{n+1}^f M_j \simeq  (L_{T(n+1)}Z) \otimes L_{n+1}^f M_j,\] since 
$Z \otimes L_{n+1}^f M_j \simeq (L_{L_{n+1}^f M_j}Z) \otimes L_{n+1}^f M_j$.

A helpful tool for our argument is the natural equivalence 
\[Z^{hH} \simeq L_{T(n+1)}(Z_{hH}) = L_{T(n+1)}(\colim_H Z),\] 
where $Z$ is any $T(n+1)$-local spectrum with an 
action by a finite group $H$ (\cite[Theorem 1.5]{KuhnInventiones}, partly \cite{MahowaldShick}; there are helpful presentations of this result in \cite[Section 1]{shorttelescopic}, \cite[page 350]{KuhnInventiones}). The action of $H$ on $Z$ induces a diagram 
$BH \to \mathrm{Sp}_{T(n)}$, where $\mathrm{Sp}_{T(n)}$ is the $\infty$-category of 
$T(n)$-local spectra, and in the above equivalence, the rightmost expression is the colimit 
of this diagram.

{\it By} (H1), the map $(\colim_{i \geq 0} K(\eta \scriptstyle{\nearrow} 
\mspace{1.5mu}\displaystyle{}i)) \otimes T(n+1)$ is an equivalence, which implies that for each $j$, the map $\colim_{i \geq 0} (K(\eta \scriptstyle{\nearrow} 
\mspace{1.5mu}\displaystyle{}i) \otimes L_{n+1}^fM_j)$ is an equivalence: that is, the canonical 
maps \[(L_{T(n+1)}K(E_n^{hG_n})) \otimes L_{n+1}^fM_j \to \colim_{i \geq 0} 
((L_{T(n+1)}K(E_n^{hU_i}))^{hG_n/U_i} \otimes L_{n+1}^fM_j)\] are equivalences. This 
gives the last step in the equivalences
\begin{align*}
L_{T(n+1)}K(L_{K(n)}S^0) & \simeq 
\lim_{j \geq 0} (K(E_n^{hG_n}) \otimes L_{n+1}^fM_j) 
\\ & \simeq \lim_{j \geq 0} ((L_{T(n+1)}K(E_n^{hG_n})) \otimes L_{n+1}^fM_j)
\\ & \simeq 
\lim_{j \geq 0} \colim_{i \geq 0} ((L_{T(n+1)} K(E_n^{hU_i}))^{hG_n/U_i} \otimes L_{n+1}^fM_j). 
\end{align*}

Now we let $i$ and $j$ be fixed non-negative integers and consider 
\[(L_{T(n+1)} K(E_n^{hU_i}))^{hG_n/U_i} \otimes L_{n+1}^fM_j.\] The homotopy fixed points and $(-) \otimes L_{n+1}^fM_j$ commute to yield the natural 
equivalence
\[(L_{T(n+1)} K(E_n^{hU_i}))^{hG_n/U_i} \otimes L_{n+1}^fM_j 
\simeq ((L_{T(n+1)} K(E_n^{hU_i})) \otimes L_{n+1}^fM_j)^{hG_n/U_i},\] because 
the homotopy orbits and $(-) \otimes L_{n+1}^fM_j$ commute. In more 
detail, we have
\begin{align*}
(L_{T(n+1)} & K(E_n^{hU_i}))^{hG_n/U_i} \mspace{-1mu} \otimes \mspace{-1mu} L_{n+1}^fM_j
\simeq 
L_{T(n+1)}(\colim_{G_n/U_i} L_{T(n+1)} K(E_n^{hU_i})) \mspace{-1mu}\otimes\mspace{-1mu} L_{n+1}^fM_j
\\ & \mspace{0mu} \simeq 
 (\colim_{G_n/U_i} L_{T(n+1)} K(E_n^{hU_i})) \otimes L_{n+1}^fM_j 
 \\ & \mspace{0mu}
 \simeq
 L_{T(n+1)}((\colim_{G_n/U_i} L_{T(n+1)} K(E_n^{hU_i})) \otimes L_{n+1}^fM_j)\\ & \mspace{0mu} 
\simeq
L_{T(n+1)}(\colim_{G_n/U_i} ((L_{T(n+1)} K(E_n^{hU_i})) \otimes L_{n+1}^fM_j))
\\ & \mspace{0mu}\simeq 
((L_{T(n+1)} K(E_n^{hU_i})) \otimes L_{n+1}^fM_j)^{hG_n/U_i}.
\end{align*} 
Since $((L_{T(n+1)} K(E_n^{hU_i})) \otimes L_{n+1}^fM_j)^{hG_n/U_i} 
\simeq (K(E_n^{hU_i}) \otimes L_{n+1}^fM_j)^{hG_n/U_i},$ we obtain the natural equivalence
\[(L_{T(n+1)} K(E_n^{hU_i}))^{hG_n/U_i} \otimes L_{n+1}^fM_j
\simeq \mathrm{Tot}(C(n, j, i)^\bullet).\]

Putting the conclusions of the last two paragraphs together and then pushing further, we obtain
\begin{align*}
L_{T(n+1)} & K(L_{K(n)}S^0) \simeq 
\lim_{j \geq 0} \colim_{i \geq 0} ((L_{T(n+1)} K(E_n^{hU_i}))^{hG_n/U_i} \otimes L_{n+1}^fM_j)
\\ & \simeq \lim_{j \geq 0} \colim_{i \geq 0} \mathrm{Tot}(C(n, j, i)^\bullet) 
\simeq
\lim_{j \geq 0} \mathrm{Tot}(\colim_{i \geq 0} C(n, j, i)^\bullet)\\
& \simeq \lim_{j \geq 0} \bigl(\colim_{i \geq 0} (K(E_n^{hU_i}) \otimes L_{n+1}^fM_j)\bigr)^{hG_n} 
=
(L_{T(n+1)}K(E_n))^{hG_n},
\end{align*} 
where the third equivalence 
{\it applies} (H2) and the fourth equivalence uses (\ref{model}). 

Now we prove Theorem \ref{moreprogressmaybe}. We continue the conventions used 
above, but now we define $T(n+1) := v_{n+1}^{-1}F(n+1)$, where $F(n+1)$ is an atomic 
$\mu$-spectrum. We assume the validity of the following two statements:
\begin{itemize}
\item[{$\mathrm{(H1}^{\prime}\mathrm{)}$}] The map $\displaystyle{\colim_{i \geq 0} K(\eta} \scriptstyle{\nearrow} \mspace{1.5mu}\displaystyle{}i)$ is a $T(n+1)$-equivalence. 
\item[{$\mathrm{(H2}^{\prime}\mathrm{)}$}]
The canonical map $\displaystyle{\colim_{i \geq 0} \displaystyle{} \mathrm{Tot}(C(n, i)^\bullet)
\to \mathrm{Tot}(\colim_{i \geq 0} C(n,i)^\bullet)}$ is an equivalence, 
where for each $i$, $C(n,i)^\bullet := \textstyle{\prod}_{(G_n/U_i)^\bullet} (K(E_n^{hU_i}) \otimes T(n+1))$.
\end{itemize} We use each assumption just once and we mark the occurrences with 
``{\it by} $\mathrm{(H1}^{\prime}\mathrm{)}$" and ``{\it by} $\mathrm{(H2}^{\prime}\mathrm{)}$." 

We have
\begin{align*}
(K(E_n) \otimes T(n+1))^{hG_n} & \simeq 
\mathrm{Tot}(\colim_{i \geq 0} \textstyle{\prod}_{(G_n/U_i)^\bullet} (K(E_n^{hU_i}) \otimes T(n+1)))\\
& \simeq  \colim_{i \geq 0} \mathrm{Tot}(\textstyle{\prod}_{(G_n/U_i)^\bullet} (K(E_n^{hU_i}) \otimes T(n+1))) \\ &\simeq \colim_{i \geq 0} (K(E_n^{hU_i}) \otimes T(n+1))^{hG_n/U_i},
\end{align*} where the first equivalence applies (\ref{model}) and the second equivalence 
is {\it by} $\mathrm{(H2}^{\prime}\mathrm{)}$. Since $F(n+1)$ is an atomic $\mu$-spectrum, 
$Z \otimes T(n+1)$ is $T(n+1)$-local for any spectrum $Z$, by Remark \ref{ring2}. Thus, 
as in the above proof of Theorem \ref{wouldgiveprogress}, there is a natural 
equivalence 
\[(K(E_n^{hU_i}) \otimes T(n+1))^{hG_n/U_i} \simeq (L_{T(n+1)}K(E_n^{hU_i}))^{hG_n/U_i} 
\otimes T(n+1),\] for each $i$. We now conclude that
\begin{align*}
(K(E_n) & \otimes T(n+1))^{hG_n} \simeq  \colim_{i \geq 0} (K(E_n^{hU_i}) \otimes T(n+1))^{hG_n/U_i}\\ & \simeq \colim_{i \geq 0} ((L_{T(n+1)}K(E_n^{hU_i}))^{hG_n/U_i} 
\otimes T(n+1))\\
& \simeq (\colim_{i \geq 0} (L_{T(n+1)}K(E_n^{hU_i}))^{hG_n/U_i}) 
\otimes T(n+1) \\ & \simeq L_{T(n+1)}K(L_{K(n)}S^0) \otimes T(n+1) 
\simeq K(L_{K(n)}S^0) \otimes T(n+1),\end{align*} where the 
penultimate step is {\it by} $\mathrm{(H1}^{\prime}\mathrm{)}$.

\bibliographystyle{plain}

\end{document}